\documentclass[12pt,reqno]{amsart}
\usepackage[english]{babel}
\usepackage{amsthm}
\usepackage{inputenc}
\usepackage{amssymb}

\newtheorem{thm}{Theorem}[section]

\newtheorem{lem}[thm]{Lemma}
\newtheorem{prop}[thm]{Proposition}
\newtheorem{defn}[thm]{Definition}

\newtheorem{exam}[thm]{Example}
\DeclareMathOperator{\gr}{gr}
\begin{document}

\title[Naturally graded Zinbiel algebras]{The classification of naturally graded Zinbiel algebras with characteristic sequence equal to $(n-p,p)$}

\author{J.K. Adashev}
\address{[J.K. Adashev] Institute of Mathematics, National University of Uzbekistan, Tashkent, 100125, Uzbekistan}
\email{adashevjq@mail.ru}
\author{M. Ladra}
\address{[M. Ladra] Department of Algebra, University of Santiago de Compostela, 15782, Spain.}
\email{manuel.ladra@usc.es}
\author{B.A. Omirov}
\address{[B.A. Omirov] Institute of Mathematics, National University of Uzbekistan, Tashkent, 100125, Uzbekistan}
\email{omirovb@mail.ru}

\thanks{The last two authors were supported by Ministerio de Econom\'ia y Competitividad (Spain), grant MTM2013-43687-P (European FEDER support included).
The second author was also supported by Xunta de Galicia, grant GRC2013-045 (European FEDER support included).}

\begin{abstract}
This work is a continuation of the description of some classes of nilpotent Zinbiel algebras.
 We focus on the study of Zinbiel algebras with restrictions to gradation and characteristic sequence.
  Namely, the classification of naturally graded Zinbiel algebras with characteristic sequence equal to $ (n-p, p)$ is obtained.
\end{abstract}


\maketitle

\section{Introduction}

This paper is devoted to investigation of algebras, which are Koszul dual to Leibniz algebras. These algebras  were introduced in the middle of 90-th of the last century by the
 French mathematician J.-L. Loday \cite{L} and they are called Zinbiel algebras (Leibniz written in reverse order).

A crucial fact of the theory of finite dimensional Zinbiel algebras is the nilpotency of such algebras over a field of zero characteristic \cite{DT}. Since  the description
of finite-dimensional complex Zinbiel algebras is a boundless problem (even if they are nilpotent), their study should be carried out by adding some additional restrictions (on index of nilpotency, gradation, characteristic sequence, etc).

In general, investigation of Zinbiel algebras goes parallel to the study of nilpotent Leibniz algebras.
For instance, $n$-dimensional Leibniz algebras of nilindices $n+1$ and $n$ (which is equivalent to admit characteristic sequences equal to $(n)$ and $(n-1,1)$, respectively) were described in papers \cite{Sib} and \cite{Fil}. Similar description for Zinbiel algebras were obtained in the paper \cite{AOK}.

In the study of $n$-dimensional Leibniz algebras of nilindex $n-1$ (see \cite{2-Fil}) it was noted that
characteristic sequences of such algebras are equal to either $(n-2, 1,1)$ or $(n-2, 2)$. Description of Leibniz (Zinbiel) algebras with such characteristic sequence were obtained in \cite{Quasi-Fil} and \cite{2-Fil} (respectively, \cite{AOK2}).

Later on, naturally graded Leibniz algebras of nilindex $n-2$ that admit the following characteristic sequences
\[(n-3, 3), \quad (n-3, 2, 1), \quad (n-3, 1, 1, 1)\]
 were investigated in a series of papers \cite{n-3,Red,p-Fil}, respectively. Description of naturally graded Zinbiel algebras with these properties was given in \cite{A} and \cite{n-3Zin}.

Finally, the latest progress in the description of the structure of nilpotent Leibniz algebras was obtained in papers \cite{p-Fil}  and \cite{Kam}.
 In particular, naturally graded nilpotent $n$-dimensional Leibniz algebras with characteristic sequences equal to $(n-p, p)$ and $(n-p, 1 \dots, 1)$ were described.
  Since the description of $p$-filiform Zinbiel algebras (that are Zinbiel algebras with characteristic sequence equal to $(n-p, 1 \dots, 1)$) was obtained in \cite{Zin},
   in order to complete the description similar to \cite{p-Fil}, in this paper we present the description (up to isomorphism)
    of naturally graded Zinbiel algebras with characteristic sequence equal to $(n-p, p)$.

All considered algebras and vector spaces in this work are assumed to be finite dimensional and complex. In order to keep tables of multiplications of algebras short, we will omit zero products.

\section{Preliminaries}

In this section we give definitions and known results necessary to proceed further to the main part of the work.

\begin{defn}
An algebra $A$ over a field $F$ is called a Zinbiel algebra if for any $x, y, z\in A$ the following identity holds:
\[(x\circ y)\circ z=x\circ(y\circ z)+x\circ(z\circ y),\]
where $\circ$ is the multiplication of the algebra $A$.
\end{defn}

For an arbitrary Zinbiel algebra we define the \emph{lower series} as follows:
\[A^1 = A, \quad A^{k+1} = A\circ A^k,  \quad  k\geq1.\]

\begin{defn} A Zinbiel algebra $A$ is called nilpotent if there exists $s\in \mathbb{N}$ such that $A^s=0$. The minimal such number is called the nilindex of $A$.
\end{defn}

\begin{defn} An $n$-dimensional Zinbiel algebra $A$ is called null-filiform if $\dim A^i = (n + 1) - i$ for $1 \leq i
\leq n + 1$.
\end{defn}

It is clear by definition that an algebra $A$ being null-filiform is equivalent to admitting the maximal possible nilindex.

Let $x$ be an element of the set $A\setminus A^2$. For an operator of a left multiplication $L_x$ (defined as  $L_x(y)=x\circ y$) we define a descending sequence $C(x)=(n_1, n_2, \dots, n_k)$, where $n =
n_1+n_2+\dots+n_k$, which consists of the sizes of Jordan blocks of the operator $L_x$. On the set of such sequences we consider the lexicographical order, that is, $(n_1, n_2, \dots, n_k)\leq
(m_1, m_2, \dots, m_s)$ if there exists $i\in N$ such that $n_j=m_j$ for all $j<i$ and $n_i<m_i$.

\begin{defn} The sequence $C(A)=\mathop {\max
}\limits_{x \in A\setminus A^2}C(x)$ is called the characteristic sequence of the algebra $A$.
\end{defn}

\begin{exam}
Let $C(A) = (1, 1, \dots, 1)$. Then  the algebra $A$ is abelian.
\end{exam}
\begin{exam}
An $n$-dimensional Zinbiel algebra $A$ is null-filiform if and only if $C(A) = (n)$.
\end{exam}

Let $A$ be a finite-dimensional Zinbiel algebra of nilindex $s$. We set $A_i:=A^i/A^{i+1}, \ \ 1 \leq i \leq s-1$, and $\gr A: = A_1 \oplus A_2\oplus \cdots\oplus
A_{s-1}$. From the condition $A_i\circ A_j\subseteq A_{i+j}$ we derive a graded algebra $\gr A$. The graduation constructed in a such way is called the natural graduation.
 If a Zinbiel algebra $A$ is isomorphic to the  algebra $\gr A$, then the algebra $A$ is called a naturally graded Zinbiel algebra.

Further we need the following lemmas.

\begin{lem}[\cite{G}]\label{lem2} For any $n, a \in \mathbb{N} $
the following equality holds
\[ \sum\limits_{k = 0}^n {( - 1)^k C_a^k C_{a + n - k - 1}^{n - k} }
= 0.\]
\end{lem}

\begin{lem}[\cite{DT}]\label{lem1}Let $A$ be a Zinbiel algebra with the following products to be known
\[e_1\circ e_i = e_{i+1}, \quad 1\leq i\leq k-1.\]
Then
\[e_i\circ e_j =C_{i + j -1}^je_{i+j},  \quad  2\leq i+j \leq k,\]
where $C_a^b=\binom{a}{b}$ denotes the binomial coefficient.
\end{lem}

\section{Main results}

Let $A$ be a Zinbiel algebra and $C(A)=(n_1, n_2, \dots, n_k)$ its characteristic sequence. Then there exists a basis $\{e_1, e_2, \dots, e_n\}$ such that the matrix of an operator of the left multiplication by the element $e_1$ has the form
\[
L_{e_1, \sigma}=\begin{pmatrix}
J_{n_{\sigma(1)}}& 0&\dots&0 \\
0& J_{n_{\sigma (2)}}&\dots& 0  \\
\vdots&\vdots&\vdots&\vdots\\
0&0&\dots&J_{n_{\sigma(s)}}
\end{pmatrix},
\]
where $\sigma(i)$ belongs to $\{1, 2, \dots, s\}$.

By a suitable permutation of basis elements we can assume that $n_{\sigma(2)}\geq
n_{\sigma(3)}\geq\cdots \geq n_{\sigma(s)}$.

Let $A$ be a naturally graded Zinbiel algebra with characteristic sequence equal to $(n_1, n_2, \dots, n_k)$.

\begin{prop}\label{prop3} There is no naturally graded Zinbiel algebra with $n_{\sigma(1)} = 1$ and $n_{\sigma(2)}\geq 4$.
\end{prop}
\begin{proof}
From the condition of the proposition we have the products:
\[\begin{array}{lll}
e_1\circ e_1 = 0, &  e_1\circ e_i = e_{i+1}, & 2\leq i\leq
4,\\[1mm]
e_1\circ e_i = e_{i+1},  & \sum\limits_{k=1}^t{n_{\sigma
(k)}}+1\leq i \leq \sum\limits_{k=1}^{t+1}{n_{\sigma (k)}}-1, &
3\leq t\leq s-1,\\[1mm]e_1\circ e_i=0, & i=\sum\limits_{k=1}^t{n_{\sigma
(k)}},& 3\leq t\leq s.\end{array}\]

By using the property of Zinbiel algebras:
\[(a\circ b)\circ c = (a\circ c)\circ b,\]
we obtain
\[e_3\circ e_1 = (e_1\circ e_2)\circ e_1 = (e_1\circ e_1)\circ
e_2=0 \Rightarrow e_3\circ e_1 = 0.\]

The chain of equalities
\[0 = (e_1\circ e_1)\circ e_3=e_1\circ(e_1\circ e_3)+e_1\circ(e_3\circ e_1)=e_1\circ e_4 =
e_5\] implies $e_5 = 0$, that is, we get a contradiction with the condition $n_{\sigma(2)}\geq 4$ which completes the proof of the proposition.
\end{proof}

The next example shows that the condition
$n_{\sigma(2)}\geq 4$ of Proposition~\ref{prop3} is essential.

\begin{exam}
Let $A$ be a four-dimensional Zinbiel algebra with the table of multiplication:
\[e_1\circ e_2 = e_3, \quad e_1\circ e_3 = e_4,  \quad e_2\circ e_1 =-e_3. \]

Then $C(A)=(3, 1)$ and the matrix of the  operator of the left multiplication on $e_1$ has the form: $\begin{pmatrix}
J_1& 0 \\
0&J_3 \end{pmatrix} $.
\end{exam}
Let $A$ be an arbitrary Zinbiel algebra with characteristic sequence equal to $(n-p, p)$. Then the matrix of  the operator of the left multiplication by $e_1$ admits one of the following forms:
\[ I.
\begin{pmatrix}
J_{n-p}&0 \\
0&{J_p}
\end{pmatrix};
\qquad II. \begin{pmatrix}
J_p&0 \\
0&J_{n-p}
\end{pmatrix}, \ \ \ \ n\geq 2p.\]

\begin{defn} A Zinbiel algebra is called an algebra of the first type (of type $I$) if the operator $L_{e_1}$ has the form
$\begin{pmatrix}
J_{n-p}&0 \\
0&J_p
\end{pmatrix}$;  otherwise it is called an algebra of the second type (of type $II$).
\end{defn}

Taking into account results of papers \cite{AOK2}  and \cite{AOK}, we will consider only $n$-dimensional naturally graded Zinbiel algebras with $C(A) = (n-p,p)$,  $p\geq 3$.

\subsection{Classification of Zinbiel algebras of type $I$}

Let $A$ be a Zinbiel algebra of type $I$. Then we have the existence of a basis $\{e_1, e_2, \dots, e_{n-p}, f_1, f_2, \dots, f_{p} \}$
such that the products containing an element $e_1$ on the left are as follows:
\[e_1\circ e_i = e_{i+1}, \ \ 1\leq i\leq n-p-1.\]

From Lemma~\ref{lem1} we obtain
\begin{equation}\label{eq1}\begin{array}{lll}
 e_i\circ e_j = C_{i+j-1}^je_{i+j}, & 2\leq i+j\leq n-p, & e_1\circ e_p =
0,\\[2mm]
e_1\circ f_i =f_{i+1}, & 1\leq i\leq p-1, & e_1\circ f_{p} = 0.\\
 \end{array} \end{equation}

It is easy to see that \[A_1 =\langle e_1, f_1\rangle, \ A_2
=\langle e_2, f_2\rangle, \ \dots, \ A_p =\langle e_p,
f_p\rangle, \ A_{p+1} =\langle e_{p+1}\rangle, \  \dots, \
A_{n-p}=\langle e_{n-p}\rangle.\]

Let
\begin{equation}\label{eq2}\begin{array}{ll}
 f_1\circ e_i = \alpha_ie_{i+1}+\beta_if_{i+1}, & 1\leq i \leq
p-1,\\
 f_1\circ e_i = \alpha_ie_{i+1}, & p\leq i \leq n-p-1,\\
 f_1\circ f_i =\gamma_ie_{i+1}+\delta_if_{i+1}, &  1\leq i\leq
p-1,\\
f_1\circ f_p =\gamma_pe_{p+1}. &\\
\end{array}\end{equation}

\begin{prop}\label{prop1} Let $A$ be a Zinbiel algebra of type $I$.
Then for the structural constants $\alpha_i, \beta_i,  \gamma_i$
and $\delta_i$ we have the following restrictions:
\[\left\{\begin{array}{ll}
\alpha_{i+1}=0, & 1\leq i \leq
n-p-2,\\[1mm]
\beta_{i+1}=\prod\limits_{k=0}^{i}\frac{k+\beta_1}{k+1}, & 1\leq i \leq p-2,\\[1mm]
(i+1)\gamma_i=\beta_1\left(2\gamma_1+\sum\limits_{k=2}^{i}\gamma_k\right), & 1\leq i\leq n-p-2,\\[1mm]
(i+\beta_1)\delta_i=\beta_1\left(2\delta_1+\sum\limits_{k=2}^{i}\delta_k\right),
& 1\leq i\leq p-2.
\end{array}\right.\]
\end{prop}

\begin{proof} First, we calculate the products $f_i\circ
e_1$ and $f_2\circ e_i$.

Consider
\[f_2\circ e_1
=e_1\circ(f_1\circ e_1)+e_1\circ(e_1\circ
f_1)=\alpha_1e_3+(1+\beta_1)f_3.\]

Using the chain of equalities
\[f_i\circ e_1=(e_1\circ f_{i-1})\circ
e_1 =e_1\circ(f_{i-1}\circ e_1)+e_1\circ(e_1\circ f_{i-1}),\]
we deduce $f_i\circ e_1 =\alpha_1e_{i+1}+(i-1+\beta_1)f_{i+1}$ for
$1\leq i \leq p-1$ and $f_i\circ e_1 =\alpha_1e_{i+1}$ for $p\leq i
\leq n-p-1$.

From the equality
\[e_i\circ f_1=(e_1\circ e_{i-1})\circ f_1 =e_1\circ(e_{i-1}\circ
f_1)+e_1\circ(f_1\circ e_{i-1}),\] we obtain $e_i\circ
f_1=\sum\limits_{k=1}^{i-1}\alpha_k
e_{i+1}+\sum\limits_{k=0}^{i-1}\beta_{k}f_{i+1}$ for $1\leq i\leq
p-1$ and $e_i\circ f_1=\sum\limits_{k=1}^{i-1}\alpha_{k}e_{i+1}$ for
$p\leq i\leq n-p-1$.

Consider
\[f_2\circ e_2
=e_1\circ(f_1\circ e_2)+e_1\circ(e_2\circ
f_1)=(\alpha_1+\alpha_2)e_4+(1+\beta_1+\beta_2)f_4.\]

From $f_2\circ e_i=(e_1\circ f_{1})\circ e_i =e_1\circ(f_{1}\circ
e_i)+e_1\circ(e_i\circ f_{1})$, we have
\[\begin{array}{ll}
f_2\circ
e_i=\sum\limits_{k=1}^{i}\alpha_ke_{i+2}+\sum\limits_{k=0}^{i}\beta_kf_{i+2},
& 1\leq i\leq p-2, \\[1mm]
f_2\circ e_i=\sum\limits_{k=1}^{i}\alpha_ke_{i+2}, & p-1\leq i\leq
n-p-2.
\end{array}\]

Now we calculate the products $f_i\circ f_1$ and $f_2\circ f_i$.

We have
\[f_2\circ f_1 = (e_1\circ f_{1})\circ f_1
=2\gamma_1e_3+2\delta_1f_3, \qquad f_3\circ f_1 = (e_1\circ
f_{2})\circ f_1 =(2\gamma_1+\gamma_2)e_4+(2\delta_1+\delta_2)f_4.\]

By induction we obtain
\[\begin{array}{ll}
f_i\circ f_1
=(2\gamma_1+\sum\limits_{k=2}^{i-1}\gamma_{k})e_{i+1}+(2\delta_1+\sum\limits_{k=2}^{i-1}\delta_{k})f_{i+1},&2\leq
i \leq p-1,\\[1mm]
f_{p}\circ f_1
=(2\gamma_1+\sum\limits_{k=2}^{p-1}\gamma_{k})e_{p+1}.&\end{array}\]

Similarly, from $f_2\circ f_i=(e_1\circ f_{1})\circ f_i
=e_1\circ(f_{1}\circ f_i)+e_1\circ(f_i\circ f_{1})$ we derive
\[\begin{array}{ll}
f_2\circ f_i
=(2\gamma_1+\sum\limits_{k=2}^{i}\gamma_{k})e_{i+2}+(2\delta_1+\sum\limits_{k=2}^{i}\delta_{k})f_{i+2},&
2\leq i \leq p-2,\\[1mm]
f_2\circ f_{i}
=(2\gamma_1+\sum\limits_{k=2}^{i}\gamma_{k})e_{i+2}, &p-1\leq i
\leq p.\end{array}\]

If $\beta_1=1$, then from the equality $(f_1\circ f_{1})\circ e_1
=(f_1\circ e_{1})\circ f_1$ we get
\[2\gamma_1(1-\beta_1)=\alpha_1^2-\delta_1\alpha_1,\ \
(1-\beta_1)\delta_1=\alpha_1(1+\beta_1).\]
Consequently, $\alpha_1=0$.

Let $\beta_1\neq1$. Then taking the following change:
\[e_1^\prime=e_1, \qquad f_1^\prime
=\frac{\alpha_1}{\beta_1-1}e_1+f_1,\] we obtain $\alpha_1^\prime=0$.

Form the equalities
\[f_1\circ e_{i+1} = f_1\circ (e_1\circ e_i) = (f_1\circ e_1)\circ e_i - if_1\circ e_{i+1} =\beta_1f_2\circ e_i-if_1\circ e_{i+1},\]
we derive $(i+1)f_1\circ e_{i+1} =\beta_1f_2\circ e_i$. Therefore,
\[(i+1)\alpha_{i+1}e_{i+2}+(i+1)\beta_{i+1}f_{i+2}=\beta_1\left(\sum\limits_{k=1}^{i}\alpha_ke_{i+2}+
\sum\limits_{k=0}^{i}\beta_kf_{i+2}\right).\]

Comparing coefficients at the basis elements and applying induction, we deduce
\[\begin{array}{ll}
\alpha_{i+1}=0, & 1\leq i \leq
n-p-2,\\[3mm]
\beta_{i+1}=\prod\limits_{k=0}^{i}\frac{k+\beta_1}{k+1},&1\leq i
\leq p-2.\end{array}\]

Considering the equality $(f_1\circ f_{i})\circ e_1 =(f_1\circ e_{1})\circ
f_i$ leads to the rest of the restrictions of the proposition.
\end{proof}

In the next proposition we calculate the products $e_i\circ
f_j$ and $f_j\circ e_i$.

\begin{prop} Let $A$ be a Zinbiel algebra of type $I$.
Then the following expressions are true:

\begin{equation}\label{eq3}e_i\circ
f_j=\sum\limits_{k=0}^{i-1}{C_{i+j-2-k}^{j-1}\beta_k}f_{i+j}, \ \
\text{for} \ \ 2\leq i+j\leq p, \end{equation}
\begin{equation}\label{eq4}f_i\circ
e_j=\sum\limits_{k=0}^j{C_{i+j-2-k}^{i-2}\beta_k}f_{i+j}, \ \
\text{for} \ \ 2\leq i+j\leq p, \end{equation} where $\beta_0=1$.
\end{prop}

\begin{proof} We shall prove \eqref{eq3}-\eqref{eq4} by induction. From \eqref{eq1} and
\eqref{eq2} we get the correctness of \eqref{eq3} and \eqref{eq4} for $i = 1$.

Consider \eqref{eq4} for $j=1$. Since $f_1\circ e_1 =\beta_1f_2$ and $f_2\circ e_1 =e_1\circ(f_1\circ e_1)+e_1\circ(e_1\circ f_1)=(1+\beta_1)f_3$, then using the equalities
$f_i\circ e_1=(e_1\circ f_{i-1})\circ e_1 =e_1\circ(f_{i-1}\circ e_1)+e_1\circ(e_1\circ f_{i-1})$ and induction, we deduce $f_i\circ e_1 =(i-1+\beta_1)f_{i+1}$ for $1\leq i \leq p-1$.

From
\[e_i\circ f_1=(e_1\circ e_{i-1})\circ f_1 =e_1\circ(e_{i-1}\circ
f_1)+e_1\circ(f_1\circ e_{i-1}),\] it implies $e_i\circ
f_1=\sum\limits_{k=0}^{i-1}\beta_kf_{i+1}$ for $1\leq i\leq p-1$.
Therefore, the equalities \eqref{eq3} are true for $j=1$ and arbitrary $i$.

Let us suppose that expressions \eqref{eq3}-\eqref{eq4} are true for $i$ and any value of $j$. The proof of the expressions for $i+1$ is obtained by the following chain of equalities:
\begin{align*}
e_{i+1}\circ f_j & =e_1\circ(e_i\circ
f_j)+e_1\circ(f_j\circ
e_i)\\
&=e_1\circ\left(\sum\limits_{k=0}^{i-1}{C_{i+j-2-k}^{j-1}\beta_k
}f_{i+j}+\sum\limits_{k =0}^i{C_{i+j-2-k}^{j-2}\beta_k
}f_{i+j}\right)\\
& =\sum\limits_{k=0}^{i-1}{C_{i+j-2-k}^{j-1}\beta_k}f_{i+j+1}+\sum\limits_{k
=0}^i{C_{i+j-2-k}^{j-2}\beta_k}f_{i+j+1}\\
& =\left({\sum\limits_{k=0}^{i-1}{C_{i+j-2-k}^{j-1}\beta_k}+\sum\limits_{k
=0}^i{C_{i+j-2-k}^{j-2}\beta_k}}\right)f_{i+j+1}=\sum\limits_{k=0}^i{C_{i+j-1
-k}^{j-1}\beta_k}f_{i+j+1}.
\end{align*}

Here we used the well-known formula $C_n^{m-1}+C_n^m=C_{n+1}^m$.

The proof of expressions \eqref{eq4} is analogous.
\end{proof}

Below, we clarify the restrictions on structural constants of the algebra with relation to the dimension and  the parameter $\beta_1$.
\begin{prop}\label{prop5} Let $A$ be a Zinbiel algebra of type $I$. Then the following restrictions are true:

(1) Case $\dim A\geq 2p+1$.

If $\beta_1\neq1$ then \[\left\{\begin{array}{ll} \gamma_i=0,&
1\leq i\leq p-1, \\[1mm]
\delta_i=0,& 1\leq i\leq p-2,\\[1mm]
(p-1+\beta_1)\gamma_p=0, &\\[1mm]
(p-2+\beta_1)\delta_{p-1}=0.\end{array}\right.\]

If $\beta_1=1$ then
\[\left\{\begin{array}{ll}
\beta_i=1,& 1\leq i\leq p-1, \\[1mm]
\gamma_i=\gamma_1,&
1\leq i\leq p, \\[1mm]
\delta_i=\delta_1,& 1\leq i\leq p-1.\\[1mm]
\end{array}\right.\]

(2) Case $\dim A=2p$.

If $\beta_1\neq1$ then
\[\left\{\begin{array}{ll} \gamma_i=0,&
1\leq i\leq p-2, \\[1mm]
\delta_i=0,& 1\leq i\leq p-2,\\[1mm]
(p-2+\beta_1)\gamma_{p-1}=0, &\\[1mm]
(p-2+\beta_1)\delta_{p-1}=0.\end{array}\right.\]

If $\beta_1=1$ then
\[\left\{\begin{array}{ll}
\beta_i=1,& 1\leq i\leq p-1, \\[1mm]
\gamma_i=\gamma_1,&
1\leq i\leq p-1, \\[1mm]
\delta_i=\delta_1,& 1\leq i\leq p-1.\\[1mm]
\end{array}\right.\]
\end{prop}

\begin{proof} Let $\dim A\geq2p+1$. Then from Proposition~\ref{prop1} we have
\begin{equation}\label{eq5}\left\{\begin{array}{ll}
(i+1)\gamma_i=\beta_1\left(2\gamma_1+\sum\limits_{k=2}^{k}\gamma_k\right), & 1\leq i\leq p-1,\\[1mm]
(i+\beta_1)\delta_i=\beta_1\left(2\delta_1+\sum\limits_{k=2}^{k}\delta_k\right),
& 1\leq i\leq p-2.
\end{array}\right.\end{equation}

Consider
\[(f_1\circ f_i)\circ e_1=f_1\circ(f_i\circ e_1)+f_1\circ(e_1\circ f_i)=(i+\beta_1)(\gamma_{i+1}e_{i+2}+\delta_{i+1}f_{i+2}).\]

On the other hand,
\[(f_1\circ f_i)\circ e_1=(\gamma_{i}e_{i+1}+\delta_{i}f_{i+1})\circ
e_1=(i+1)\gamma_{i}e_{i+2}+(i+\beta_1)\delta_{i}f_{i+2}.\] Hence,
\begin{equation}\label{eq6}\left\{\begin{array}{ll}
(i+1)\gamma_{i}=(i+\beta_1)\gamma_{i+1},& 1\leq i \leq p-1,\\[1mm]
(i+\beta_1)\delta_{i}=(i+\beta_1)\delta_{i+1}, & 1\leq i \leq
p-2.\end{array}\right.\end{equation}

Considering the cases $\beta_1\neq1$ and $\beta_1=1$ together with the expressions
\eqref{eq5} and \eqref{eq6} leads to the restrictions of the case $\dim A \geq2p+1$. The proof of the remaining case is carried out in a similar fashion.
\end{proof}

Consider a general change of basis of the algebra $A$. It is known that for  naturally graded Zinbiel algebras it is sufficient to take the change of basis in the form:
\[e_1^\prime=Ae_1+Bf_1,\qquad f_1^\prime=Ce_1+Df_1,\]
where $AD-BC\neq0$.

\begin{prop}\label{prop7} Let $A$ be a Zinbiel algebra of type $I$ and let $\beta_1\neq1$. Then
\[\left\{\begin{array}{ll}
e_{i+1}^\prime=e_1^\prime\circ e_i^\prime,& 1\leq i\leq n-p-1,\\[1mm]
f_{i+1}^\prime=e_1^\prime\circ f_i^\prime,& 1\leq i\leq p-1,\\[1mm]
e_i^\prime=A^ie_i+A^{i-1}B\sum\limits_{k=0}^{i-1}\beta_{k}f_i, & 1\leq i\leq p-1,\\[1mm]
f_i^\prime=A^{i-1}Df_i, & 1\leq i\leq p-1,\\[1mm]
C=0.&\\[1mm]
\end{array}\right. \]
\end{prop}

\begin{proof} From $f_1^\prime\circ f_1^\prime =0$ we get $C=0$. The proof of the proposition is completed by considering products $e_1^\prime\circ e_i^\prime=e_{i+1}^\prime$ and
$e_1^\prime\circ f_i^\prime=f_{i+1}^\prime$.
\end{proof}

\begin{thm} \label{thm3.7} Let $A$ be an $n$-dimensional ($n\geq2p+2$) Zinbiel algebra of type $I$ and
with characteristic sequence equal to $(n-p, p)$. Then it is isomorphic to one of the following non-isomorphic algebras:

\[A_1:\left\{\begin{array}{ll} e_i\circ e_j
=C_{i+j-1}^je_{i+j}, & 2\leq i+j\leq n-p,\\[1mm] e_i\circ f_j =
\sum\limits_{k=0}^{i-1}{C_{i+j-2-k}^{j-1}\beta_k}f_{i+j},  \ \
f_i\circ e_j=\sum\limits_{k=0}^j{C_{i+j-2-k}^{i-2}
\beta_k}f_{i+j}, & 2\leq i+j\leq p,
\end{array}\right. \]
 where $\beta_{i+1}=\prod\limits_{k=0}^{i}\frac{k+\beta_1}{k+1} \ \ \mbox{for} \ \ 1\leq i \leq p-2$
and $\beta_1\in \mathbb{C}$;

\[A_2:\left\{\begin{array}{lll} e_i\circ
e_j = C_{i+j-1}^je_{i+j}, & 2\leq i+j\leq n-p,& f_1\circ f_{p-1} =f_{p},\\[1mm]
e_i\circ f_j =
\sum\limits_{k=0}^{i-1}{C_{i+j-2-k}^{j-1}\beta_k}f_{i+j}, &
f_i\circ e_j=\sum\limits_{k=0}^j{C_{i+j-2-k}^{i-2}
\beta_k}f_{i+j}, & 2\leq i+j\leq p,\\[1mm]
\end{array}\right. \] where $\beta_{i}=(-1)^iC_{p-2}^i \ \mbox{for} \ \ 1\leq i \leq p-2$;

\[A_3:\left\{\begin{array}{ll}
e_i\circ e_j = C_{i+j-1}^je_{i+j}, & 2\leq i+j\leq n-p,\\[1mm]
e_i\circ f_j =f_i\circ e_j=f_i\circ f_j =C_{i+j-1}^{j}f_{i+j}, &
2\leq i+j\leq p.\end{array}\right.\]
\end{thm}

\begin{proof} From Proposition~\ref{prop5} for $\beta_1\neq1$ we obtain a table of multiplication of the algebra:

\begin{align*}
e_i\circ e_j  &= C_{i+j-1}^je_{i+j}, \ &   & \qquad    \qquad  \qquad  \quad  \qquad 2\leq i+j\leq n-p,\\
f_1\circ f_i  & =\delta_if_{i+1},  \quad  1\leq i\leq
p-1,  & f_i\circ f_j & =\varphi(\delta_1, \delta_2, \dots,
\delta_s)f_{i+j},  \quad 2\leq i+j\leq p,\\
e_i\circ f_j & = \sum\limits_{k=0}^{i-1}{C_{i+j-2-k}^{j-1}\beta_k}f_{i+j},
 \ \  & f_i\circ e_j &=\sum\limits_{k=0}^j{C_{i+j-2-k}^{i-2}
\beta_k}f_{i+j}, \quad 2\leq i+j\leq p,
\end{align*}
where
$(p-1+\beta_1)\gamma_{p}=0, \ \ (p-2+\beta_1)\delta_{p-1}=0$.

Consider
\[(f_1\circ f_{p})\circ e_1 =f_1\circ (f_{p}\circ e_1)+f_1\circ (e_1\circ f_{p})
=0.\]

On the other hand,
\[(f_1\circ f_{p})\circ e_1=\gamma_pe_{p+1}\circ e_1 =(p+1)\gamma_pe_{p+2}.\]
Therefore, we deduce $\gamma_p=0$.

Applying Proposition~\ref{prop7} to the general change of basis, from the equalities
\begin{align*}
f_1^\prime\circ f_{p-1}^\prime &=(Df_1)\circ
(A^{p-2}Df_{p-1})=A^{p-2}D^2\delta_{p-1}f_p \\
&= \delta_{p-1}^\prime
f_{p}^\prime
=\delta_{p-1}^\prime(A^{p-1}D+A^{p-2}BD\delta_{p-1})f_p,
\end{align*}
we get
$\delta_{p-1}^\prime=\frac{D\delta_{p-1}}{A+B\delta_{p-1}}$.

If $\delta_{p-1}=0$ then $\delta_{p-1}^\prime=0$, and we obtain the algebra $A_1$.

If $\delta_{p-1}\neq0$ then by choosing $D=\frac{A+B\delta_{p-1}}{\delta_{p-1}}$ and from $(p-2+\beta_1)\delta_{p-1}=0$, we have $\delta_{p-1}^\prime=1, \
\beta_1=2-p$, that is, we have the algebra $A_2$.

In case of $\beta_1=1$ we have $\delta_i=\delta_1, \ \ 1 \leq i\leq
p-1$. Putting $D=\frac{A+B\delta_{1}}{\delta_{1}}$ we obtain
$\delta^\prime_{1} =1$. Consequently, $f_1\circ f_i =f_{i+1}$ for $1\leq i\leq p-1$. From Lemma~\ref{lem1} we deduce $f_i\circ f_j =C_{i+j-1}^{j}f_{i+j}, \ \ 2\leq i+j\leq p$. Thus, we get the algebra $A_3$.
\end{proof}

In the following theorem the classification for $n=2p+1$ is presented.

\begin{thm} \label{thm3.8} Let $A$ be a Zinbiel algebra of type $I$ and with characteristic sequence equal to $(p+1, p)$. Then it is isomorphic to one of the following non-isomorphic algebras:

\[A_4:\left\{\begin{array}{lll}
e_i\circ e_j = C_{i+j-1}^je_{i+j}, & 2\leq i+j\leq p+1, &f_1\circ f_{p} =e_{p+1},\\[1mm]
e_i\circ f_j
=\sum\limits_{k=0}^{i-1}{C_{i+j-2-k}^{j-1}\beta_k}f_{i+j}, &
f_i\circ e_j=\sum\limits_{k=0}^j{C_{i+j-2-k}^{i-2}
\beta_k}f_{i+j},  & 2\leq i+j\leq p, \\[1mm]
 \end{array}\right.\]
where $\beta_0=1, \ \ \beta_{i}=(-1)^iC_{p-1}^i$ for $1\leq i \leq p-1$;

\[A_5:\left\{\begin{array}{ll}
e_i\circ e_j = C_{i+j-1}^je_{i+j}, & 2\leq i+j\leq p+1,\\[1mm]
e_i\circ f_j =
\sum\limits_{k=0}^{i-1}{C_{i+j-2-k}^{j-1}\beta_k}f_{i+j}, \
f_i\circ e_j=\sum\limits_{k=0}^j{C_{i+j-2-k}^{i-2}
\beta_k}f_{i+j},  & 2\leq i+j\leq p, \end{array}\right.\]
where $\beta_0=1, \ \ \beta_{i+1}=\prod\limits_{k=0}^{i}\frac{k+\beta_1}{k+1}$ for $1\leq i \leq p-2$ and $\beta_1\in \mathbb{C}$;

\[A_6:\left\{\begin{array}{lll}
e_i\circ e_j = C_{i+j-1}^je_{i+j}, & 2\leq i+j\leq p+1, & f_1\circ f_{p-1} = f_{p},\\[1mm]
e_i\circ f_j =
\sum\limits_{k=0}^{i-1}{C_{i+j-2-k}^{j-1}\beta_k}f_{i+j}, &
f_i\circ e_j=\sum\limits_{k=0}^j{C_{i+j-2-k}^{i-2}
\beta_k}f_{i+j}, & 2\leq i+j\leq p,\\[1mm]
 \end{array}\right.\]
where $\beta_0=1, \ \ \beta_{i}=(-1)^iC_{p-2}^i$ for $1\leq i \leq p-2$ and $\beta_{p-1}=0$;

\[A_7:\left\{\begin{array}{ll}
e_i\circ e_j = C_{i+j-1}^je_{i+j},  &  2\leq i+j\leq p+1,\\[1mm]
e_i\circ f_j = f_i\circ e_j=C_{i+j-1}^{j}f_{i+j},  & 2\leq i+j\leq
p,\\[1mm]
f_i\circ f_{j}
=\gamma_{1}C_{i+j-1}^{j}e_{i+j}+\delta_1C_{i+j-1}^{j}f_{i+j}, &
2\leq i+j\leq p,\\[1mm]
f_i\circ f_{j} = \gamma_{1}C_{i+j-1}^{j}e_{p+1}, &  i+j=
p+1,\end{array}\right.\]
where $\gamma_1, \delta_1\in \mathbb{C}$.
\end{thm}

\begin{proof} From Proposition~\ref{prop5} for
$\beta_1\neq1$ we obtain a table of multiplication of $A$:
\begin{align*}
e_i\circ e_j &= C_{i+j-1}^je_{i+j}, \ & &  \qquad \qquad \qquad \ \qquad \qquad 2\leq i+j\leq p+1,\\
 f_1\circ f_{p-1} &=\delta_{p-1}f_{p}, \ \ &  f_1\circ f_{p}
&=\gamma_{p}e_{p+1},\\
e_i\circ f_j &= \sum\limits_{k=0}^{i-1}{C_{i+j-2-k}^{j-1}\beta_k}f_{i+j},
\ \ &  f_i\circ e_j& =\sum\limits_{k=0}^j{C_{i+j-2-k}^{i-2}
\beta_k}f_{i+j},  \ \ 2\leq i+j\leq p,
\end{align*}
where $\beta_0=1, \ (p-1+\beta_1)\gamma_{p}=0$ and $(p-2+\beta_1)\delta_{p-1}=0$.

Consider the general change of basis as above. Then from Proposition~\ref{prop7}, we have
\[\left\{\begin{array}{l}
e_p^\prime=A^pe_p+(A^{p-1}B\sum\limits_{k=0}^{p-1}\beta_{k}+
A^{p-2}B^2\sum\limits_{k=0}^{p-2}\beta_{k}\delta_{p-1})f_p,\\[1mm]
e_{p+1}^\prime=(A^{p+1}+(A^{p-1}B^2\sum\limits_{k=0}^{p-1}\beta_{k}+
A^{p-2}B^3\sum\limits_{k=0}^{p-2}\beta_{k}\delta_{p-1})\gamma_p)e_{p+1},\\[1mm]
f_p^\prime=(A^{p-1}D+A^{p-2}BD\delta_{p-1})f_p,\\[1mm]
\end{array}\right. \]

The equality $e_1^\prime\circ f_p^\prime = 0$ in the new basis implies $B\gamma_p= 0$.

\textbf{Case 1.} Let $\gamma_{p}\neq0$. Then $B=0$ and $\beta_1=1-p$, $\delta_{p-1}=0$.

Considering the equality $f_1^\prime\circ f_{p}^\prime=\gamma_{p}^\prime
e_{p+1}^\prime$, we derive $A^2\gamma_p^\prime=D^2\gamma_p$.

Setting $D=\frac{A}{\sqrt{\gamma_p}}$, we obtain $\gamma_p^\prime=1$. Thus, we get the algebra $A_4$.

\textbf{Case 2.} Let $\gamma_{p}=0$. Then considering the equality $f_1^\prime\circ f_{p-1}^\prime= \delta_{p-1}^\prime f_{p}^\prime$, we deduce
$\delta_{p-1}^\prime=\frac{D\delta_{p-1}}{A+B\delta_{p-1}}$.

If $\delta_{p-1}=0$ then $\delta_{p-1}^\prime=0$, that is, we obtain  the algebra $A_5$.

If $\delta_{p-1}\neq0$ then $\beta_1=2-p$ and putting $D=\frac{A+B\delta_{p-1}}{\delta_{p-1}}$, we get $\delta_{p-1}^\prime=1$ and the algebra $A_6$.

Now we consider case $\beta_1=1$. Using Proposition~\ref{prop5}, we obtain the algebra $A_7$.
\end{proof}

Below, we present the classification of Zinbiel algebras with characteristic sequence equal to $C(A)=(p, p)$.

\begin{thm} Let $A$ be a Zinbiel algebra with characteristic sequence $(p, p)$. Then it is isomorphic to one of the following non-isomorphic algebras:

\[A_8:\left\{\begin{array}{ll}
e_i\circ e_j = C_{i+j-1}^je_{i+j}, & 2\leq i+j\leq p,\\[1mm]
e_i\circ f_j =
\sum\limits_{k=0}^{i-1}{C_{i+j-2-k}^{j-1}\beta_k}f_{i+j},\ \
f_i\circ e_j=\sum\limits_{k=0}^j{C_{i+j-2-k}^{i-2}
\beta_k}f_{i+j},  & 2\leq i+j\leq p, \end{array}\right.\]
where $\beta_0=1, \ \ \beta_{i+1}=\prod\limits_{k=0}^{i}\frac{k+\beta_1}{k+1}$ for $1\leq i \leq p-2$ and $\beta_1\in \mathbb{C}$;

\[A_9:\left\{\begin{array}{lll}
e_i\circ e_j = C_{i+j-1}^je_{i+j}, & 2\leq i+j\leq p, & f_1\circ f_{p-1} = f_{p},\\[1mm]
e_i\circ f_j =
\sum\limits_{k=0}^{i-1}{C_{i+j-2-k}^{j-1}\beta_k}f_{i+j}, &
f_i\circ e_j=\sum\limits_{k=0}^j{C_{i+j-2-k}^{i-2}
\beta_k}f_{i+j},  & 2\leq i+j\leq p,\\[1mm]
  \end{array}\right.\]
where $\beta_0=1, \ \ \beta_{i}=(-1)^iC_{p-2}^i$ for $1\leq i \leq p-2$ and $\beta_{p-1}=0$;

\[A_{10}:\left\{\begin{array}{lll}
e_i\circ e_j = C_{i+j-1}^je_{i+j}, & 2\leq i+j\leq p,&f_1\circ f_{p-1} = e_p+\delta_{p-1}f_{p},\\[1mm]
e_i\circ f_j =
\sum\limits_{k=0}^{i-1}{C_{i+j-2-k}^{j-1}\beta_k}f_{i+j}, &
f_i\circ e_j=\sum\limits_{k=0}^j{C_{i+j-2-k}^{i-2}
\beta_k}f_{i+j},  & 2\leq i+j\leq p,\\[1mm]
 \end{array}\right. \]
where $\beta_0=1, \ \ \beta_{i}=(-1)^iC_{p-2}^i$ for $1\leq i \leq p-2, \ \beta_{p-1}=0$ and $\delta_{p-1}\in
 \mathbb{C}$;

\[A_{11}:  \ e_i\circ e_j = C_{i+j-1}^je_{i+j},  \ \ e_i\circ f_j = f_i\circ e_j=C_{i+j-1}^{j}f_{i+j},   \ \ 2\leq i+j\leq p.\]

\[A_{12}: \left\{\begin{array}{ll} e_i\circ e_j = C_{i+j-1}^je_{i+j}, & 2\leq i+j\leq p,\\[1mm]
e_i\circ f_j = f_i\circ e_j=f_i\circ f_{j} = C_{i+j-1}^{j}f_{i+j},
&2\leq i+j\leq p. \end{array}\right.\]
\end{thm}

\begin{proof} The proof of this theorem is carried out by applying the methods and arguments as in the proof of Theorems~\ref{thm3.7} and \ref{thm3.8}.
\end{proof}

\subsection{Classification of Zinbiel algebras of type $II$.}

Consider a Zinbiel algebra of type $II$. From the condition on the operator $L_{e_1}$ we have the existence of a basis $\{e_1, e_2, \dots, e_p, f_1, f_2, \dots, f_{n-p} \}$ such that the products involving $e_1$ on the left side have the form:
\[e_1\circ e_i = e_{i+1}, \quad 1\leq i\leq
p-1.\]
Applying Lemma~\ref{lem1}, we get
\[\begin{array}{lll}
  e_i\circ e_j = C_{i+j-1}^je_{i+j}, & 2\leq i+j\leq p, & e_1\circ e_p = 0,\\
   e_1\circ f_i =f_{i+1}, & 1\leq i\leq n-p-1,&e_1\circ f_{n-p} = 0. \\
\end{array}\]

It is easy to see that \[A_1 =\langle e_1, f_1\rangle, \ A_2
=\langle e_2, f_2\rangle, \dots, \ A_p =\langle e_p, f_p\rangle,
\ A_{p+1} =\langle f_{p+1}\rangle,  \dots, \ A_{n-p}=\langle
f_{n-p}\rangle.\]

Let us introduce notations:
\[\begin{array}{lll}
 f_1\circ e_i = \alpha_ie_{i+1}+\beta_if_{i+1}, & 1\leq i \leq
p-1,& f_1\circ e_p = \beta_pf_{p+1},\\
 f_1\circ f_i =\gamma_ie_{i+1}+\delta_if_{i+1}, &  1\leq i\leq
p-1,&\\
 f_1\circ f_i =\delta_if_{i+1}, & p\leq i\leq n-p-1, & f_1\circ
f_{n-p}=0.\\
\end{array}\]

The following proposition can be proved similar to Proposition~\ref{prop1}.
\begin{prop} Let $A$ be a Zinbiel algebra of type $II$. Then for structural constants $\alpha_i, \beta_i,  \gamma_i$ and $\delta_i$ the following restrictions hold:
\[\left\{\begin{array}{ll}
\alpha_{i+1}=0 & 1\leq i \leq
p-2,\\[1mm]
\beta_{i+1}=\prod\limits_{k=0}^{i}\frac{k+\beta_1}{k+1} & 1\leq i \leq p-1,\\[1mm]
(i+1)\gamma_i=\beta_1\left(2\gamma_1+\sum\limits_{k=2}^{i}\gamma_k\right), & 1\leq i\leq p-2,\\[1mm]
(i+\beta_1)\delta_i=\beta_1\left(2\delta_1+\sum\limits_{k=2}^{i}\delta_k\right),
& 1\leq i\leq n-p-2.
\end{array}\right.\]
\end{prop}

\begin{prop} Let $A$ be a Zinbiel algebra of type $II$.
Then the following expressions hold:

\begin{equation}\label{eq7}e_i\circ
f_j=\sum\limits_{k=0}^{i-1}{C_{i+j-2-k}^{j-1}\beta_k}f_{i+j}, \
\mbox{при}  \ \ 1\leq i\leq  p,  \ \ p+1\leq i+j\leq n-p,
\end{equation}
\begin{equation}\label{eq8}f_i\circ e_j=\sum\limits_{k=0}^j{C_{i+j-2-k}^{i-2}\beta_k}f_{i+j},
\ \ \mbox{при} \ \ 1\leq j\leq p,  \ \ p+1\leq i+j\leq n-p,
\end{equation} where $\beta_0=1$.
\end{prop}

\begin{proof}  We shall prove the assertion of the proposition be induction. Clearly, the relation \eqref{eq7} is true for $i = 1$.

We have \[ f_2\circ e_1 =e_1\circ(f_1\circ
e_1)+e_1\circ(e_1\circ f_1)=(1+\beta_1)f_3.\]

Using the chain of equalities \[f_i\circ e_1=(e_1\circ f_{i-1})\circ
e_1 =e_1\circ(f_{i-1}\circ e_1)+e_1\circ(e_1\circ f_{i-1}),\] and induction, we derive $f_i\circ e_1 =(i-1+\beta_1)f_{i+1}$ for $1\leq i \leq p-1$ and $f_i\circ e_1 =(i-1+\beta_1)f_{i+1}$ for $p\leq i\leq
n-p-1$. Therefore, the relation \eqref{eq8} is true for $j = 1$.

Let us assume that the relations \eqref{eq7}-\eqref{eq8} are true for $i$ and any value of $j$. The proof of these relations for $i+1$ follows from the following chain of equalities:
\begin{align*}
e_{i+1}\circ f_j &=e_1\circ(e_i\circ
f_j)+e_1\circ(f_j\circ
e_i)\\
&=e_1\circ\left(\sum\limits_{k=0}^{i-1}{C_{i+j-2-k}^{j-1}\beta_k
}f_{i+j}+\sum\limits_{k =0}^i{C_{i+j-2-k}^{j-2}\beta_k
}f_{i+j}\right)\\
&=\sum\limits_{k=0}^{i-1}{C_{i+j-2-k}^{j-1}\beta_k}f_{i+j+1}+\sum\limits_{k
=0}^i{C_{i+j-2-k}^{j-2}\beta_k}f_{i+j+1}\\
&=\left({\sum\limits_{k=0}^{i-1}{C_{i+j-2-k}^{j-1}\beta_k}+\sum\limits_{k
=0}^i{C_{i+j-2-k}^{j-2}\beta_k}}\right)f_{i+j+1}=\sum\limits_{k=0}^i{C_{i+j-1
-k}^{j-1}\beta_k}f_{i+j+1}.
\end{align*}
Checking the  correctness of the remaining relations of the proposition is analogous.
\end{proof}

Similar to the case of Zinbiel algebra of type $I$ for algebras of type $II$ we obtain the restrictions on structure constants with relation to parameter $\beta_1$.

\begin{prop} Let $A$ be a Zinbiel algebra of type $II$. Then

if $\beta_1\neq1$ then
\[\left\{\begin{array}{ll} \gamma_i=0,&
1\leq i\leq p-2, \\[1mm]
\delta_i=0,& 1\leq i\leq n-p-2,\\[1mm]
(p-2+\beta_1)\gamma_{p-1}=0, &\\[1mm]
(n-p-2+\beta_1)\delta_{n-p-1}=0.\end{array}\right.\]

if $\beta_1=1$ then
\[\left\{\begin{array}{ll}
\beta_i=1,& 1\leq i\leq p, \\[1mm]
\gamma_i=\gamma_1,&
1\leq i\leq p-1, \\[1mm]
\delta_i=\delta_1,& 1\leq i\leq n-p-1.\\[1mm]
\end{array}\right.\]
\end{prop}

In the next theorem we prove that there is no $n$-dimensional Zinbiel algebras of type $II$ with
$n\geq3p+2$.

\begin{thm} There is no Zinbiel algebras of type $II$  with characteristic sequence equal to $(n-p,
p)$ for $n\geq 3p+2$.
\end{thm}

\begin{proof} Consider for $1\leq i\leq p+1$ equalities
\[0=(e_1\circ e_p)\circ f_i=e_1\circ(e_p\circ f_i)+e_1\circ(f_i\circ e_p).\]

Applying the relations \eqref{eq7}, \eqref{eq8} and arguments similar to the ones that are used in the proof of Proposition~\ref{prop3}, we derive the relation

\begin{equation}\label{eq9}\sum\limits_{k=0}^p{C_{p+i-1-k}^{i-1}\beta_k}=0,\end{equation}
where $\beta_0=1$ and $1\leq i\leq n-p-1$.

Now we consider the determinant of the matrix of order $p+1:$
\[M=\begin{vmatrix}
1&1&1&\dots&1&1\\
C_{p+1}^1 & C_p^1 & C_{p-1}^1&\dots& C_2^1 & 1 \\
C_{p+2}^2 & C_{p+1}^2 & C_p^2 &\dots& C_3^2 &1 \\
\vdots&\vdots&\vdots&\vdots&\vdots&\vdots \\
C_{2p-1}^{p-1} & C_{2p-2}^{p-1} &C_{2p-3}^{p-1} &\dots& C_p^{p-1} &1 \\
C_{2p}^p & C_{2p-1}^p & C_{2p-2}^p &\dots& C_{p+1}^p & 1
 \end{vmatrix}
\]

Taking into account identity $C_n^{m-1}+C_n^m=C_{n+1}^m$ and subtracting from each row the previous one we obtain
\[M=\begin{vmatrix}
1&1&1&....&1&1 \\
C_p^1&C_{p-1}^1&C_{p-2}^1&\dots&1&0 \\
C_p^2&C_{p-1}^2&C_{p-2}^2&\dots&0&0 \\
\vdots&\vdots&\vdots&\vdots&\vdots&\vdots\\
C_p^{p-1}&1&0&\dots&0&0 \\
1&0&0&\dots&0&0
\end{vmatrix}=-1.\]

Since $M=-1$, the system of equations \eqref{eq9} for $i=p+1$ has only trivial solution with respect to unknown variables $\beta_i$. In particular,
$\beta_0=0$. However, $\beta_0=1$, that is, we get a contradiction to the condition
$i=p+1\leq n-p-1$, which implies the  non existence of an algebra under  the condition $n\geq 3p+2$.
\end{proof}

Let $A$ be an $n$-dimensional algebra with a basis $\{e_1, e_2, \dots, e_p, f_1, f_2, \dots, f_{n-p} \}$ and the table of multiplication
\[\begin{array}{lll}
 e_i\circ e_j=C_{i+j-1}^je_{i+j}, & 2\leq i+j\leq p,& e_1\circ
e_p=0,\\[1mm]
 e_1\circ f_{n-p}=0, & f_i\circ f_j=0, & 1\leq i, j \leq n-p,\\[1mm]
 e_i\circ f_j=\sum\limits_{k=0}^{i-1}{C_{i+j-2-k}^{j-1}\beta_k}
f_{i+j}, & f_i\circ
e_j=\sum\limits_{k=0}^j{C_{i+j-2-k}^{i-2}\beta_k}f_{i+j}, &2\leq
i+j\leq n-p,\\
\end{array}\] where $\beta_i=(-1)^i C_p^i, \ \ 0\leq i\leq p$.

It is easy to check that this algebra is a Zinbiel algebra.

The correctness of the relation \eqref{eq9} for parameters $\beta_i$ for $n=3p+1$ follows from Lemma~\ref{lem2}.
Thus, the condition $n\geq 3p+2$ is necessary.

We list the next theorems on the description of Zinbiel algebras of type $II$ without proofs. It can be carried out by applying similar arguments that were used above.

\begin{thm} A Zinbiel algebra of type $II$  with characteristic sequence equal to $(p+1,
p)$ is isomorphic to one of the following non-isomorphic algebras:

\[\widetilde{A_1}:\left\{\begin{array}{ll}
e_i\circ e_j = C_{i+j-1}^je_{i+j}, & 2\leq i+j\leq p,\\[1mm]
e_i\circ f_j =
\sum\limits_{k=0}^{i-1}{C_{i+j-2-k}^{j-1}\beta_k}f_{i+j}, \ \
f_i\circ e_j=\sum\limits_{k=0}^j{C_{i+j-2-k}^{i-2}
\beta_k}f_{i+j},  & 2\leq i+j\leq p+1,\end{array}\right.\] where
$\beta_0=1, \ \
 \beta_{i+1}=\prod\limits_{k=0}^{i}\frac{k+\beta_1}{k+1}$
for $1\leq i\leq p-1, \ \ \beta_1\in\{-p, -(p-1), \dots, -2, -1\}$;
\[\widetilde{A_2}:\left\{\begin{array}{lll} e_i\circ e_j = C_{i+j-1}^je_{i+j}, & 2\leq i+j\leq p, & f_1\circ f_{p-1}=e_{p}, \\[1mm]
e_i\circ f_j =
\sum\limits_{k=0}^{i-1}{C_{i+j-2-k}^{j-1}\beta_k}f_{i+j}, &
f_i\circ e_j=\sum\limits_{k=0}^j{C_{i+j-2-k}^{i-2}
\beta_k}f_{i+j},  & 2\leq i+j\leq p+1,\\[1mm]
 \end{array}\right.\]
where $\beta_i=(-1)^iC_{p-2}^i$ for $0\leq i\leq p-2$  и $\beta_{p-1}=\beta_p=0$;
\[\widetilde{A_3}:\left\{\begin{array}{lll}e_i\circ e_j = C_{i+j-1}^je_{i+j}, & 2\leq i+j\leq p, &  f_1\circ f_{p}=f_{p+1},\\[1mm]
e_i\circ f_j =
\sum\limits_{k=0}^{i-1}{C_{i+j-2-k}^{j-1}\beta_k}f_{i+j}, &
f_i\circ e_j=\sum\limits_{k=0}^j{C_{i+j-2-k}^{i-2}
\beta_k}f_{i+j},  & 2\leq i+j\leq p+1,\\[1mm]
 \end{array}\right.\]
where $\beta_i=(-1)^iC_{p-1}^i$  for $0\leq i\leq p-1$  and $\beta_p=0$;

\[\widetilde{A_4}:\left\{\begin{array}{ll} e_i\circ e_j = C_{i+j-1}^je_{i+j}, & 2\leq i+j\leq p,\\[1mm]
e_i\circ f_j = f_i\circ e_j=C_{i+j-1}^{j}f_{i+j},  & 2\leq i+j\leq
p+1. \end{array}\right.\]
\[\widetilde{A_5}:\left\{\begin{array}{ll} e_i\circ e_j = C_{i+j-1}^je_{i+j}, & 2\leq i+j\leq p,\\[1mm]
e_i\circ f_j = f_i\circ e_j=C_{i+j-1}^{j}f_{i+j},  \  \ f_i\circ
f_{j} = C_{i+j-1}^{j}f_{i+j},  & 2\leq i+j\leq p+1.
\end{array}\right.\]
\end{thm}

\begin{thm} A Zinbiel algebra of type $II$ with characteristic sequence equal to $(p+2, p)$ is isomorphic to one of the following non-isomorphic algebras:

\[\widetilde{A_6}:\left\{\begin{array}{ll} e_i\circ e_j = C_{i+j-1}^je_{i+j}, & 2\leq i+j\leq p,\\[1mm]
e_i\circ f_j =
\sum\limits_{k=0}^{i-1}{C_{i+j-2-k}^{j-1}\beta_k}f_{i+j},  & 1\leq
i\leq p, \ \ 2\leq i+j\leq p+2,\\[1mm]
f_i\circ e_j=\sum\limits_{k=0}^j{C_{i+j-2-k}^{i-2}
\beta_k}f_{i+j},  & 1\leq j\leq p, \ \ 2\leq i+j\leq
p+2,\end{array}\right.\]
where $\beta_0=1, \ \
 \beta_{i+1}=\prod\limits_{k=0}^{i}\frac{k+\beta_1}{k+1}$
 for $1\leq i\leq p-1, \ \ \beta_1\in\{-p, -(p-1), \dots, -2, -1\}$;
\[\widetilde{A_7}:\left\{\begin{array}{ll}
 e_i\circ e_j = C_{i+j-1}^je_{i+j}, & 2\leq i+j\leq p,  \ \ f_1\circ f_{p-1}=e_{p},\\[1mm]
 e_i\circ f_j = \sum\limits_{k=0}^{i-1}{C_{i+j-2-k}^{j-1}\beta_k}f_{i+j},  & 1\leq i\leq p, \ \ 2\leq i+j\leq
p+2,\\[1mm]
f_i\circ e_j=\sum\limits_{k=0}^j{C_{i+j-2-k}^{i-2}
\beta_k}f_{i+j},  & 1\leq j\leq p, \ \ 2\leq i+j\leq
p+2,\end{array}\right.\]
 where $\beta_i=(-1)^iC_{p-2}^i$ for $0\leq i\leq p-2$  и $\beta_{p-1}=\beta_p=0$;
\[\widetilde{A_8}:\left\{\begin{array}{ll} e_i\circ e_j = C_{i+j-1}^je_{i+j}, & 2\leq i+j\leq p,  \ \ f_1\circ f_{p+1}=f_{p+2},\\[1mm]
e_i\circ f_j =
\sum\limits_{k=0}^{i-1}{C_{i+j-2-k}^{j-1}\beta_k}f_{i+j},  & 1\leq
i\leq p, \ \ 2\leq i+j\leq p+2,\\[1mm]
f_i\circ e_j=\sum\limits_{k=0}^j{C_{i+j-2-k}^{i-2}
\beta_k}f_{i+j},  & 1\leq j\leq p, \ \ 2\leq i+j\leq
p+2,\end{array}\right.\]
where $\beta_i=(-1)^iC_{p}^i$ for $0\leq i\leq p$.
\end{thm}

\begin{thm} A Zinbiel algebra of type $II$ with characteristic sequence equal to $(p+t, p)$, for $3\leq t\leq p+1$, is isomorphic to one of the following non-isomorphic algebras:
\[\widetilde{A_9}: \left\{\begin{array}{ll}
e_i\circ e_j = C_{i+j-1}^je_{i+j}, & 2\leq i+j\leq p,\\[1mm]
e_i\circ f_j =
\sum\limits_{k=0}^{i-1}{C_{i+j-2-k}^{j-1}\beta_k}f_{i+j},  & 1\leq
i\leq p, \ \ 2\leq i+j\leq p+t,\\[1mm]
f_i\circ e_j=\sum\limits_{k=0}^j{C_{i+j-2-k}^{i-2}
\beta_k}f_{i+j},  & 1\leq j\leq p, \ \ 2\leq i+j\leq
p+t,\end{array}\right.\]
where $\beta_0=1, \ \
 \beta_{i+1}=\prod\limits_{k=0}^{i}\frac{k+\beta_1}{k+1}$
 for $1\leq i\leq p-1, \ \ \beta_1\in\{-p, -(p-1), \dots, -(t-1)\}$;
\[\widetilde{A_{10}}: \left\{\begin{array}{ll}
e_i\circ e_j = C_{i+j-1}^je_{i+j}, & 2\leq i+j\leq p,  \ \
f_1\circ f_{p-1}=e_{p},\\[1mm]
e_i\circ f_j =
\sum\limits_{k=0}^{i-1}{C_{i+j-2-k}^{j-1}\beta_k}f_{i+j},
 & 1\leq i\leq p, \ \ 2\leq i+j\leq p+t,\\[1mm]
 f_i\circ e_j=\sum\limits_{k=0}^j{C_{i+j-2-k}^{i-2}
\beta_k}f_{i+j},  & 1\leq j\leq p, \ \ 2\leq i+j\leq
p+t,\end{array}\right.\]
 where $\beta_i=(-1)^iC_{p-2}^i$ for $0\leq i\leq p-2$  and $\beta_{p-1}=\beta_p=0$.
\end{thm}

\end{document}